\documentclass[11pt]{amsart}
\usepackage[margin=3cm]{geometry}               
\usepackage[usenames,dvipsnames,svgnames,table]{xcolor}    
\usepackage{graphicx}
\usepackage{tikz}
\usepackage{tikz-cd}
\usepackage{cases}
\usetikzlibrary{decorations.pathreplacing}
\usepackage{lineno, hyperref}
\usepackage{amssymb,amsfonts, amsmath, amsthm}
\usepackage{epstopdf}
\usepackage[shortlabels]{enumitem}
\usepackage{ulem}
\usepackage{cancel}
\usepackage{thmtools} 
\usepackage{mathtools} 
\usepackage{mathrsfs}
\usepackage{yhmath}

\title{A bi-Lipschitz invariant for analytic function germs}

\newtheorem{thm}{Theorem}[section]
\newtheorem{lem}[thm]{Lemma}

\newtheorem{example}[thm]{Example}
\newtheorem{rem}[thm]{Remark}

\numberwithin{equation}{section}

\newcommand{\bb}{\mathbb}
\newcommand{\al}{\mathcal}

\newcommand{\ord}{{\rm ord}}
\newcommand{\Inv}{\mathbf{\rm Inv}}
\newcommand{\grad}{{\rm grad}}

\newcommand{\hot}{{\rm h.o.t}}

\newcommand{\Horn}{{\rm Horn}}
\newcommand{\gr}{{\rm gr}}

\author[Nhan Nguyen]{Nguyen Xuan Viet Nhan}
\address{FPT University, Danang, Vietnam}
\email{nguyenxuanvietnhan@gmail.com}

\begin{document}

\maketitle
\begin{abstract} 
In this paper, we introduce a new bi-Lipschitz invariant for analytic function germs in two variables, enhancing the Henry-Parusinski invariant. 
\end{abstract}


\normalem 

\section{Introduction}
Let $f, g: (\bb C^2, 0) \to (\bb C, 0)$ be analytic germs. We call $f$ and $g$ \textit{bi-Lipschitz equivalent} if there exists a germ of a bi-Lipschitz homeomorphism $h: (\bb C^2, 0) \to (\bb C^2, 0)$ such that $ f = g \circ h$. This paper focuses on finding bi-Lipschitz invariants for Lipschitz equivalence, a central problem in Lipschitz geometry of singularities. Many significant results have been established in this area;  see, for example, \cite{Trotman}, \cite{HP}, \cite{HP2}, \cite{Fukui}, \cite{PT}, and \cite{nrs}. Among these, the invariant introduced by Henry and Parusiński \cite{HP, HP2} is regarded as the most refined, as it demonstrates that bi-Lipschitz equivalence allows moduli. A typical example is the family of function germs defined by
$$
f_t(x, y) = x^3 -3t^2xy^4 + y^6.
$$
Using Henry-Parusi\'nski’s invariant, one can deduce that for distinct values $t$ and $t'$ (with $t$ close to $t'$) the function germs $f_t$ and $f_{t'}$ belong to different bi-Lipschitz equivalence classes. This invariant is a powerful tool, however, as shown in \cite{PT}, it is not sufficient to fully characterize the bi-Lipschitz type.

In \cite{FR}, Fernandes and Ruas studied the following family of germs:
$$
g_t(x, y) = x^3 - 3t^2xy^{10} + y^{12}.
$$
They proved that Henry-Parusi\'nski's invariant is constant  while this family is not strongly bi-Lipschitz equivalent, meaning there is no Lipschitz vector field that trivializes $g_t$. An interesting question is whether this family is bi-Lipschitz trivial.

In this paper, we introduce an improved version of Henry–Parusi\'nski's invariant that can effectively distinguish the bi-Lipschitz  equivalence classes of $g_t$ (see Example \ref{EX1}). Our approach builds upon the recent work of P\u{a}unescu and Tib\u{a}r \cite{PT} on gradient canyons under bi-Lipschitz equivalence (see also \cite{PT2}). This invariant also uncovers germs with Lipschitz modality $\geq 2$ (Example 3.2), providing a significant step toward a comprehensive classification of Lipschitz unimodal germs.

\textbf{Notation and Conventions:} We use $|\cdot|$ to denote the Euclidean norm. To compare the asymptotic behavior of functions $\varphi$ and $\psi$ near 0, we employ the standard notations $\varphi = o(\psi)$ and $\varphi = O(\psi)$. The notation $\varphi \sim \psi$ means $\varphi = O(\psi)$ and $\psi = O(\varphi)$. We define $\mathbb{C}^* = \mathbb{C} \setminus {0}$. For a function of the form $a_1 y^{h_1} + \dots + a_k y^{h_k} + o(y^{h_k})$, where the $h_i$ are rational, we often write $a_1 y^{h_1} + \dots + a_k y^{h_k} + \text{h.o.t.}$ (higher-order terms).

\section{Main result}

\subsection{P\u{a}unescu--Tib\u{a}r's results}
In this section, we recall some results due to P\u{a}unescu and Tib\u{a}r in \cite{PT}, which will be used in the next section.

By an analytic arc we mean a convergent power series of the form
$$x = \alpha (y) = c_1 y^{n_1/N} + c_2 y^{n_2/N} + \cdots, c_i \in \bb C$$
where $N \leq n_1 < n_2 < \ldots$ are positive integers having no common divisor.  There are $N$ conjugates of $\alpha$ that are of the following forms: 

$$\alpha_{conj}^{(j)} = \sum c_i \theta^{jn_i} y ^{n_i /N}, \hspace{0.5cm}$$
where $\theta$ is an $N$-th root of the unity, $j = 0, \ldots, N-1$.  

We call $n_1/N$ the order of $\alpha$, denoted by $\ord_y(\alpha) = n_1/N$.

 We denote by $\alpha_*$ the image of the map $t \mapsto (\alpha(t^N), t^N)$.  Note that $\alpha_* = \alpha_{conj\: *}^{(j)}$ for every $0\leq j\leq N-1$.

Given two analytic arcs $\alpha$ and $\beta$, the contact order between $\alpha$ and $\beta$, and between $\alpha_*$ and $\beta_*$  are defined by
\begin{itemize}
\item $\ord_0(\alpha, \beta) =  \ord_y(\alpha(y) - \beta(y))$
\item $\ord_0(\alpha_*, \beta_*) = \max \ord_y (\tilde{\alpha}(y) -  \tilde{\beta}(y))$
where the maximum is taken over all conjugates $\tilde{\alpha}$  of $\alpha$ and $\tilde{\beta}$ of $\beta$. 
\end{itemize}
It is obvious that $\ord_0(\alpha, \beta) \leq \ord_0(\alpha_*, \beta_*)$.

Let $f: (\bb C^2,0) \to (\bb C,0)$ be a germ of analytic function with Taylor expansion
$$ f(x,y) = H_k(x, y) + H_{k + 1}(x, y) + \ldots$$
where $H_i$ is a homogeneous polynomial of degree $i$. We assume that $f$ is \textit{a regular} in $x$, meaning $H_k(1, 0) \neq 0$. We also assume that $f$ does not have multiple root.

 A \textit{polar arc} of $f$ is an analytic arc $x = \alpha(y)$ such that  $\frac{\partial f}{\partial x}  (\alpha(y), y) = 0$. A conjugate of $\alpha$ is considered as  
a distinct polar arc from $\alpha$ if $\theta \neq 1$. 

For a polar arc $\alpha$, the gradient degree $d_{\gr}(\alpha)$ is defined as follows:

\begin{itemize}
    \item If $f(\alpha(y), y) \not\equiv 0$ then  $d_{\gr}(\alpha)$ is the smallest number $q$ such that
 $$\ord_y (\|\grad f(\alpha(y), y)\|) = \ord_y (\|\grad f(\alpha(y) + c y^q, y)\|),$$
for a generic $c \in \bb C$. 
\item If $f(\alpha(y), y) \equiv 0$, we set $d_{\gr}(\alpha) = +\infty$.
\end{itemize}
It is straightforward to verify that the gradient degree is invariant under conjugation, i.e., $d_{\gr}(\alpha) = d_{\gr} (\alpha_{conj}^{(j)})$ for all $j = 0, \ldots, N-1$. 

In \cite[Section 3.2]{PT} a method was introduced to compute $d_{\gr}(\alpha)$ using the Newton polygon.





\begin{example}\rm 
Let $f_t(x, y) = x^3 - 3t^2 x y^{10} + y^{12}$. 

Then $$\frac{\partial f_t}{ \partial x} = 3 x^2 - 3t^2 y^{10}$$ $$\frac{\partial f_t}{\partial y} =  - 30t^2x y^9 + 12y^{11}$$
There are two polar arcs $ x =\alpha(y) =  ty^5$ and $ x =\beta(y)=  -ty^5$. From definition, it is easy to check that $$ d_{\gr} (\alpha) =  d_{\gr} (\beta) = 11/2.$$ 
\end{example}

Let $\alpha$ be a polar arc of $f$.  The \textit{gradient canyon} of $\alpha_*$ is the set
$$
\al{GC}(\alpha_*) = \{ \beta_* : \beta(y) = \alpha(y) + c y^{d_{\gr}(\alpha)} + \hot ,\hspace{0.3cm} c \in \mathbb{C} \},
$$
and  $d_{\gr}(\alpha_*): = d_{\gr}(\alpha)$  is called \textit{the degree of the canyon}.

We say that $\alpha$ is tangential if $\alpha_*$ is tangent to the singular locus $\Sigma_f$ of the tangent cone of the zero locus of $f$, which occurs if and only if  $d_{\gr}(\alpha) > 1$.
Note that 
$$\Sigma_f = \{(x,y) \in \bb C^2: \frac{\partial H_k}{\partial x} =\frac{\partial H_k}{ \partial y} = 0\}.$$
Moreover, if $d_{\gr}(\alpha)>1$ then $\al{GC}(\alpha_*)$ is \textit{minimal}, meaning that for any polar arc $\gamma$,   if $\al {GC}(\gamma_{*})\subset \al{GC}(\alpha_*)$ then  $\al{GC}(\gamma_{*}) =  \al{GC}(\alpha_*)$ (see \cite[Theorem 2.4]{PT}).

For an analytic arc $\gamma_*$, and positive real numbers $e \geq 1$, $\rho$, and $\epsilon \geq 0$, we define 
$$\Horn^{(e)} (\gamma_*, \rho, \epsilon )  = \bigcup_{\beta} \{\beta_*\cap B(0,\epsilon) | \beta(y) = J^e(\gamma) (y) + c y^e, |c|\leq \rho\}$$
where $J^e(\gamma)$ is the truncation of $\gamma$ up to power $e$.

Set
 $$ \al D^{(e)}_{\gamma_*, \rho} (\lambda, \epsilon) = \{f= \lambda\}\cap \Horn^{(e)} (\gamma_*, \rho, \epsilon ),$$
for $0< \lambda \ll \epsilon$.

\begin{thm}[{\cite[Section 5.2]{PT}}]\label{thm_pt}
Let $f, g: (\mathbb{C}^2, 0) \rightarrow (\mathbb{C}, 0)$ be analytic function germs that are mini-regular in $x$ and have no multiple roots. Suppose that there exists a bi-Lipschitz homeomorphic germ $\varphi: (\mathbb{C}^2, 0) \rightarrow (\mathbb{C}^2, 0)$ such that $f = g \circ \varphi$. Let $\alpha$ be a polar arc of $f$ with $d = d_{\gr}(\alpha) > 1$. Then, there exists a polar arc $\alpha'$ of $g$ satisfying the following conditions:
\begin{enumerate}
    \item  $d_{\gr}(\alpha') = d_{\gr}(\alpha) = d$

\item Moreover, given $\epsilon$ sufficiently small and $\rho>0$, there are $0<\rho_1 < \rho_2$  such that
$$ \al D^{(d)}_{\alpha'_*, \rho_1} (\lambda, \epsilon) \subset \varphi\left(\al D^{(d)}_{\alpha_*, \rho} (\lambda, \epsilon)\right) \subset \al D^{(d)}_{\alpha'_*, \rho_2} (\lambda, \epsilon)  $$
for all $0 <\lambda \ll \epsilon$.
\end{enumerate}
\end{thm}


\subsection{New invariant} Let us begin by explaining how we arrive at the new invariant. The formal statement will be presented at the end of the section.

Let $f, g: (\bb C^2, 0) \to (\bb C,0)$ be given as in Theorem \ref{thm_pt}. Let 
 $\alpha$ and $\beta$ polar arcs of $f$ satisfying 
 \begin{enumerate}
     \item $\ord_0(\alpha, \beta) >1$
     \item $\al{GC}(\alpha_{*}) \neq  \al{GC}(\beta_*)$ 
 \end{enumerate}
 The condition (1) implies that  $\alpha$ and $\beta$ are tangential, so there is a line $L\subset \Sigma_f$ such that $\alpha_*$ and $\beta_*$ are both tangent to $L$.

 We write
\begin{equation}\label{eq11}
    f(\alpha(y), y) = a_0 y^{h_0} + a_1 y^{h_1} + \dots + O(y^{h_0 + d  -1})
\end{equation}
and

\begin{equation}\label{eq12}
    f(\beta(y), y) = b_0 y^{l_0} + b_1 y^{l_1} + \dots + O(y^{l_0 + d' -1})
\end{equation}
where  $a_0, b_0 \neq 0$, $d = d_{\gr}(\alpha)$ and $d' = d_{\gr}(\beta)$.

  \begin{rem}\label{rem1}\rm
   If $f(\alpha(y), y) = a y^h + \hot$, then for any $\gamma = \alpha(y) + t y^{m} + \hot$ with $m > \ord_y(\alpha)$, we have  
   $$
   f(\gamma(y), y) = ay^h + \dots + u(t) y^{h+m-1} + \hot,  
   $$
   where $u(t)$ is the first term that depends on $t$. The expansion of $f(\gamma(y), y)$ can be written as 
   $$
   f(\gamma(y), y) = ay^h + \dots + O(y^{h+m-1}).
   $$
   Note the latter expression is independent of perturbations of $\alpha$ of order $m$. 
\end{rem}

It follows from Theorem \ref{thm_pt} that there are polar arcs $\alpha'$ and $\beta'$ of $g$  such that
\begin{itemize}
    \item[(i)] $d_{\gr}(\alpha') =  d_{\gr}(\alpha)$,  $d_{\gr}(\beta') =  d_{\gr}(\beta)$

\item[(ii)] $\ord_0(\alpha_*, \beta_*) = \ord(\alpha'_*, \beta_*')$

\item[(iii)] $\varphi(\alpha_*) \subset  \Horn^{(d)} (\alpha'_*, \rho, \epsilon )$ and $\varphi(\beta_*) \subset  \Horn^{(d')} (\alpha'_*, \rho, \epsilon )$ as germs at the origin for some $\rho>0$ and $\epsilon>0$.
\end{itemize}
We may assume that 
\begin{equation}\label{equ12.1}
    g(\alpha'(y), y) = a'_0 y^{h'_0} + a'_1 y^{h'_1} + \ldots + O(y^{h'_0 + d  -1})
\end{equation}
and

\begin{equation}\label{equ12.2}
    g(\beta'(y), y) = b'_0 y^{l'_0} + b'_1 y^{l'_1} + \ldots + O(y^{l'_0 + d' -1})
\end{equation}
where $a'_0, b'_0 \neq 0$. 

Consider a sequence $(y_n)\subset \bb C$ with  $y_n\to 0$ as $n \to \infty$. 
From (iii), passing a subsequence if necessary,  we can write 
\begin{equation}\label{eq13}
    \varphi (\alpha(y_n), y_n) = ( \alpha' (Y_n)+ O (Y_n^d) , Y_n)
\end{equation}
and 

\begin{equation}\label{eq14}
    \varphi (\beta(y_n), y_n) = ( \beta' (\tilde{Y}_n)+  O(\tilde{Y}_n^{d'}), \tilde{Y}_n)
\end{equation}
where $(Y_n)$  and $(\tilde{Y}_n)$ are  sequences in a local coordinate. 
From Remark \ref{rem1}, it yields that 
\begin{equation}\label{eq15}
    (g\circ\varphi) (\alpha(y_n), y_n) = a'_0 Y_n^{h'_0} +  a'_1 Y_n^{h'_1} + \dots + O(Y_n^{h'_0 + d -1})
\end{equation}
and 
\begin{equation}\label{eq16}
    (g\circ\varphi) (\beta(y_n), y_n) = b'_0 \tilde{Y}_n^{l'_0} +  b'_1 \tilde{Y}_n^{l'_1} + \dots + O(\tilde{Y}_n^{l'_0 + d' -1}).
\end{equation} 

Observe that in a neighbourhood of the origin $|y_n| \sim |(\alpha(y_n), y_n)| \sim |\varphi((\alpha(y_n), y_n)| \sim |Y_n|$ and 
   $|y_n| \sim |(\beta(y_n), y_n)| \sim |\varphi((\beta(y_n), y_n)| \sim |\tilde{Y}_n|$.
   Therefore, 
   $$ |y_n| \sim |Y_n| \sim |\tilde{Y}_n|.$$
By $f = g \circ \varphi$, we have 
\begin{equation}\label{eq17}
\begin{cases}
a_0 y_n^{h_0} + a_1 y_n^{h_1} + \ldots + O(y_n^{h_0 + d  -1}) = a'_0 Y_n^{h'_0} +  a'_1 Y_n^{h'_1} + \dots + O(Y_n^{h'_0 + d -1}),\\

b_0 y_n^{l_0} + b_1 y_n^{h'_1} + \ldots + O(y_n^{l_0 + d' -1}) = b'_0 Y_2^{l'_0} +  a'_1 \tilde{Y}_n^{l'_1} + \dots + O(\tilde{Y}_n^{l'_0 + d' -1})
\end{cases}
\end{equation}
Since $|y_n| \sim |Y_n| \sim |\tilde{Y}_n|$, it yields that $ h_0 = h'_0$ and $l_0 = l'_0$. We rewrite \eqref{eq17} as follows: 

\begin{equation}\label{eq17.1}
\begin{cases}
a_0 y_n^{h_0} + a_1 y_n^{h_1} + \ldots + O(y_n^{h_0 + d  -1}) = a'_0 Y_n^{h_0} +  a'_1 Y_n^{h'_1} + \dots + O(Y_n^{h_0 + d -1}),\\

b_0 y_n^{l_0} + b_1 y_n^{h'_1} + \ldots + O(y_n^{l_0 + d' -1}) = b'_0 \tilde{Y}_n^{l_0} +  a'_1 \tilde{Y}_n^{l'_1} + \dots + O(\tilde{Y}_n^{l_0 + d' -1})
\end{cases}
\end{equation}
This implies that   
\begin{equation}\label{eq18}
\begin{cases}
\lim_{n\to \infty} \left(\frac{y_n}{Y_n}\right)^{h_0} = \frac{a'_0}{a_0} \\

\lim_{y\to 0} \left(\frac{y}{\tilde{Y}}\right)^{l_0} = \frac{b'_0}{b_0}.
\end{cases}
\end{equation}

Since $\varphi$ is bi-Lipschitz,
\begin{equation}\label{eq18(0)}
    |\varphi (\alpha(y), y) - \varphi (\beta(y), y))| \sim |(\alpha(y),y) - (\beta(y),y)| = |\alpha(y) - \beta(y)| \sim |y|^{\ord_0(\alpha, \beta)}.
\end{equation}

Note that $\ord_0(\alpha, \beta) >1$, yielding that  

$$|\varphi (\alpha(y_n), y_n) - \varphi (\beta(y_n), y_n))| \sim |y_n|^{\ord_0(\alpha, \beta)}= o(y_n).$$

It follows from \eqref{eq13} and \eqref{eq14}, $|Y_n -\tilde{Y}_n| = o (y_n) = o(Y_n)$. Alternatively, we write
\begin{equation}\label{eq18(1)}
    \tilde{Y}_n = Y_n + o(Y_n).
\end{equation}
Passing a subsequence, we may assume that $\lim_{n\to \infty} \frac{y_n}{Y_n}$ and $ \lim_{n\to \infty} \frac{y_n}{\tilde{Y}_n}$ exist. Then, by \eqref{eq18(1)} there must be an $c \in \bb C$ such that
$c = \lim_{n\to \infty} \frac{y_n}{Y_n} = \lim_{n\to \infty} \frac{y_n}{\tilde{Y}_n}$. Hence, from \eqref{eq18} we have that 
\begin{equation}\label{eq18.2}
   c^{h_0} =  \frac{a_0'}{a_0}  \text{ and }  c^{l_0}= \frac{b_0'}{b_0}.
\end{equation}

Indeed, $c$ is the constant appearing in the Henry-Parusi\'nski invariant \cite{HP}. It is solely determined by the tangent line $L$ and is independent of the specific choices of $\alpha, \alpha', \beta, \beta'$.

Let us now consider the case that $h_0 = l_0$. Then, 
$$
\frac{a'_0}{a_0} = \frac{b'_0}{b_0}.
$$

Set $\delta = \ord(\alpha, \beta)$. Note that $\delta >1$ and   $\delta \leq \ord(\alpha_*, \beta_*)$. As in \eqref{eq18(0)},
$$|\varphi(\alpha(y), y) - \varphi (\beta(y), y)|  \sim  |y|^\delta.$$
By definition,  $|Y_n - \tilde{Y}_n| \leq|\varphi(\alpha(y_n), y_n) - \varphi (\beta(y_n), y_n)| \sim |y_n|^\delta$. Thus, 
$$|Y_n - \tilde{Y}_n| = O(y_n^\delta) = O(Y_n^\delta).$$ 
We then write
$$\tilde{Y}_n = Y_n + O(Y_n^\delta).$$ 
Since $\al{GC}(\alpha_{*}) \neq  \al{GC}(\beta_*)$, so  $\delta < \min\{ d, d'\}$.  It follows that  
\begin{equation}
    b'_0 \tilde{Y}_n^{l_0} +  a'_1 \tilde{Y}_n^{l'_1} + \dots + O(\tilde{Y}_n^{l_0 + d' -1}) = b'_0 Y_n^{h_0} +  a'_1 Y_n^{l'_1} + \dots + O(Y_n^{h_0 + \delta -1})
\end{equation}
Combining with \ref{eq17.1}, we obtain
\begin{equation}\label{eq19}
\begin{cases}
a_0 y_n^{h_0} + a_1 y_n^{h_1} + \ldots + O(y_n^{h_0 + \delta  -1}) = a'_0 Y_n^{h_0} +  a'_1 Y_n^{h'_1} + \dots + O(Y_n^{h_0 + \delta-1}),\\

b_0 y_n^{h_0} + b_1 y_n^{h'_1} + \ldots + O(y_n^{h_0 + \delta -1}) = b'_0 Y_n^{h_0} +  b'_1 Y_n^{l'_1} + \dots + O(Y_n^{h_0 + \delta -1})
\end{cases}
\end{equation}
Equivalently, 
\begin{equation}\label{eq19(0)}
\begin{cases}
 y_n^{h_0} + \frac{a_1}{a_0} y_n^{h_1} + \ldots + O(y_n^{h_0 + \delta  -1}) = \frac{a'_0}{a_0} Y_n^{h_0} +  \frac{a'_1}{a_0} Y_n^{h'_1} + \dots + O(Y_n^{h_0 + \delta-1}),\\

y_n^{h_0} + \frac{b_1}{b_0} y_n^{h'_1} + \ldots + O(y_n^{h_0 + \delta -1}) = \frac{b'_0}{b_0} Y_n^{h_0} +  \frac{b'_1}{b_0} Y_n^{l'_1} + \dots + O(Y_n^{h_0 + \delta -1})
\end{cases}
\end{equation}
By $
\frac{a'_0}{a_0} = \frac{b'_0}{b_0}$, we have 
\begin{equation}\label{eq19}
\begin{cases}
 y_n^{h_0} + \frac{a_1}{a_0} y_n^{h_1} + \ldots + O(y_n^{h_0 + \delta  -1}) = \frac{a'_0}{a_0} Y_n^{h_0} +  \frac{a'_1}{a_0} Y_n^{h'_1} + \dots + O(Y_n^{h_0 + \delta-1}),\\
y_n^{h_0} + \frac{b_1}{b_0} y_n^{h'_1} + \ldots + O(y_n^{h_0 + \delta -1}) = \frac{a'_0}{a_0} Y_n^{h_0} +  \frac{b'_1}{b_0} Y_n^{l'_1} + \dots + O(Y_n^{h_0 + \delta -1})
\end{cases}
\end{equation}

Now, define \( \tilde{h} \) as follows:  
\[
\tilde{h} =
\begin{cases} 
\min\left\{ h_i \mid \frac{a_i}{a_0} \neq \frac{b_i}{b_0} \right\}, & \text{if such an index } i \text{ exists}; \\
h_0 + \delta - 1, & \text{otherwise}.
\end{cases}
\]
Similarly, define $\tilde{h}' $ as  
$$
\tilde{h}' =
\begin{cases} 
\min\left\{ h'_i \mid \frac{a'_i}{a_0} \neq \frac{b'_i}{b_0} \right\}, & \text{if such an index } i \text{ exists}; \\
h_0 + \delta - 1, & \text{otherwise}.
\end{cases}
$$

\begin{lem}\label{lem1}
    $\tilde{h} = \tilde{h}'$. 
\end{lem}

\begin{proof} 
If $\tilde{h} <h_0 + \delta -1$ and $ \tilde{h}' =h_0 + \delta -1$, we can rewrite equations in \ref{eq19} as the following forms
\begin{equation}\label{eq22}
    A + \tilde{a} y_n^{\Tilde{h}} + \hot  = \frac{a'_0}{a_0}B+ O(Y_n^{h_0 + \delta - 1})
\end{equation}
and 
\begin{equation}\label{eq23}
    A + \tilde{b} y_n^{\Tilde{h}} + \hot  = \frac{a'_0}{a_0}B+ O(Y_n^{h_0 + \delta - 1}) 
\end{equation}
where $\tilde{a} \neq \tilde{b}$.   
Subtracting side by side \eqref{eq23} by \eqref{eq22} we obtain 
\begin{equation}\label{eq24}
    (\tilde{a} - \tilde{b}) y_n^{\Tilde{h}} + \hot =    O(Y_n^{h_0 + \delta - 1})
\end{equation}
This is impossible since $|y_n|\sim |Y_n|$ and $\tilde{h} <h_0 + \delta -1$. Therefore, we can assume that both $\tilde{h}$ and $\tilde{h}'$ are strictly less than $h_0 + \delta -1$. Then, equations in \ref{eq19} can be written as: 

\begin{equation}\label{eq22(1)}
    A + \tilde{a} y_n^{\Tilde{h}} + \hot  = \frac{a'_0}{a_0}(B+ \tilde{a}' Y_n^{\Tilde{h}'}) + \hot 
\end{equation}
and 
\begin{equation}\label{eq23(1)}
    A + \tilde{b} y_n^{\Tilde{h}} + \hot  = \frac{a'_0}{a_0}(B+ \tilde{b}' Y_n^{\Tilde{h}'} ) + \hot 
\end{equation}
where $\tilde{a} \neq \tilde{b}$ and  $\tilde{a}' \neq \tilde{b}'$. 
Substracting side-by-side the above equations we get 
\begin{equation}\label{eq24}
    (\tilde{a} - \tilde{b}) y_n^{\Tilde{h}} + \hot =   \frac{a'_0}{a_0} (\tilde{a}' - \tilde{b}') Y_n^{\Tilde{h}'} + \hot 
\end{equation}
Since $|y_n| \sim |Y_n|$, it follows that  $\tilde{h} = \tilde{h}'$. The lemma is proved. 
\end{proof}

Now, suppose that $\tilde{h}< h_0 +\delta -1$. By Lemma \ref{lem1}, we have $\tilde{h}' = \tilde{h}$. Then, Equation \eqref{eq24} becomes: 

\begin{equation}\label{eq24(1)}
    (\tilde{a} - \tilde{b}) y_n^{\Tilde{h}} + \hot =   \frac{a'_0}{a_0} (\tilde{a}' - \tilde{b}') Y_n^{\Tilde{h}} + \hot 
\end{equation}
It follows that 
\begin{equation}\label{eq25}
    \frac{\Tilde{a}' - \Tilde{b}'}{\Tilde{a} - \Tilde{b}} = \frac{a_0}{a'_0}\lim_{n\to \infty}\left( \frac{y_n}{Y_n}\right)^{\Tilde{h}} = c^{\tilde{h}-h_0}
\end{equation}
where $c = \lim_{n \to  \infty} \frac{y_n}{Y_n}$ as in \eqref{eq18.2}.

We associate the pair of polar arcs $(\alpha, \beta)$ with the following numerical values:  

\begin{equation}\label{eq26}
m(\alpha, \beta) = \tilde{h}
\end{equation}

and  

\begin{equation}\label{eq27}
\nu(\alpha, \beta) = 
    \begin{cases}
        \tilde{a} - \tilde{b}, & \text{if } \tilde{h} < h_0 + \delta -1, \\
        0, & \text{otherwise}.
    \end{cases}
\end{equation}

Here, $\tilde{a}$ and $\tilde{b}$ are the coefficients of $y^{\tilde{h}}$ in the  expansions of $f(\alpha(y), y)/a_0$ and $f(\beta(y), y)/b_0$ respectively (see the left-hand sides of \eqref{eq19}).  

Note that $m(\alpha, \beta) = m(\beta, \alpha)$ and $\nu(\alpha, \beta) = - \nu(\beta, \alpha)$. These functions define a new bi-Lipschitz invariant, summarized as follows.

 Let $f: (\bb C^2, 0) \to (\bb C, 0)$ be an analytic function germ that is mini-regular in $x$ and has no multiple root. For a polar arc $\gamma: x = \gamma(y)$ of $f$, we associate to $\gamma$ to numbers $a_0(\gamma)$ and $h_0(\gamma)$ given by the expansion
$$f(\alpha(y), y) = a_0(\gamma) y^{h_0(\gamma)} + \hot$$
where $a_0(\gamma)\neq 0$.

For a line $L \subset \Sigma_f$, we denote by  $\Gamma(L)$ the set of all polar arc whose image tangent to $L$ and by $\Lambda(L)$ the set of all pairs $(\alpha, \beta)$ of polar arcs of $f$ satisfying the following conditions: 
\begin{enumerate}
    \item $\alpha_*$ and $\beta_*$ are tangent to $L$ at the origin
    \item $\al {GC}(\alpha_*) \neq \al{GC}(\beta_*)$ 
    \item $h_0(\alpha) = h_0(\beta)$
    \item $m(\alpha, \beta) < h_0(\alpha) + \ord_0(\alpha, \beta) -1$
\end{enumerate}  
where $m(\alpha, \beta)$ is defined as in \eqref{eq26}. 
Set
$$\Delta(L) = \{a_0(\gamma)y^{h_0(\gamma)}; (h_{0}(\alpha), \nu(\alpha, \beta) y^{m(\alpha, \beta)}): \gamma\in \Gamma(L),  (\alpha, \beta) \in \Lambda(L)\}/ \bb C^*$$
where  $\nu(\alpha, \beta)$ is defined as in \eqref{eq27} and the action of  $\bb C^*$ is given by: 
\begin{align*}
    \{a_{01} y^{h_{01}}, \ldots, a_{0k}y^{h_{0k}}; (l_{01}, \nu_1 y^{m_1}), \ldots, (l_{0n}, \nu_n y^{m_{n}})\} \sim \\
    \{ a_{01} c^{h_{01}} y^{h_{01}}, \ldots, a_{0k} c^{h_{0k}} y^{h_{0k}}; (l_{01}, \nu_1  c^{m_{1}-l_{01}} y^{m_{1}}), \ldots, (l_{0n},\nu_n c^{m_n-l_{0n}} y^{m_{n}})\}.
\end{align*} 

We define $\Inv^2 (f)$ as the set of all $\Delta(L)$ for $L\in \Sigma_f$. From \eqref{eq18.2} and \eqref{eq25} we obtain 
\begin{thm}
    Let $f, g: (\bb C^2, 0) \to (\bb C, 0)$  be analytic germs that are mini-regular in $x$ and have no multiple roots. If $f$ and $g$ are bi-Lipschitz equivalent then $\Inv^2(f) = \Inv^2(g)$.  
\end{thm}

\section{Examples}
Consider an analytic family of germs $f_t(x, y): (\mathbb{C}^2, 0) \to (\mathbb{C}, 0)$,  $t \in U$, where $U$ is an open set in $\mathbb{C}^k$. We say that $f_{t_1}$ and $f_{t_2}$ are not \textit{bi-Lipschitz equivalent for $t_1, t_2$ generic} if  there exists a dense subset $U'$ of $U$ such that for every $t_1, t_2$ in $U'$ and $t_1 \neq t_2$, $f_{t_1}$ is not bi-Lipschitz equivalent to $f_{t_2}$.
 
\begin{example}\label{EX1}
    Consider the family of function germs $f(x,y) = f_t (x,y) = x^3 - 3t^2 x y^{10} + y^{12}$. Then, $f_{t_1}$ and $f_{t_2}$ are not bi-Lipschitz equivalent for $t_1, t_2$ generic.
\end{example}

\begin{proof}
Consider $f_t$ with $t\neq 0$. We have  $$\frac{\partial f_t }{\partial x}(x,y) = 3x^2 -3 t^2 y^{10}$$ and $$\frac{\partial f_t}{\partial y} = -30t^2 x y^9 + 12 y^{11}.$$ 
There are two polar arcs $\alpha: x = ty^5$, $\beta: x = -ty^5$. 
And, 
$$ f_t( \alpha(y), y) = y^{12} - 2 t^3 y^{15}$$
and 
$$f_t( \beta(y), y) = y^{12} + 2 t^3 y^{15}.$$

It is easy to check that $d_{\gr}(\alpha) = d_{\gr}(\beta) = 11/2$, $\ord_0(\alpha, \beta) = 5 <11/2$. This implies that $\alpha_{*}$ and $\beta_{*}$ are in different gradient canyons of $f_{t}$. 
In addition, $h_0(\alpha) = h_0(\beta) = 12$ and $m(\alpha, \beta) = 15 < h_0(\alpha) + \ord_0(\alpha, \beta) - 1 = 16$; $\nu(\alpha, \beta) = -4t^3$ and $\nu(\beta, \alpha) = 4t^3$.
Then, 
$$\Inv^2 (f_t) = \{ y^{12}, y^{12}; (12, -4t^3y^{15}), (12, 4t^3y^{15})\}/\bb C^*.$$
It follows that if $f_{t_1}$ and $f_{t_2}$ are bi-Lipschitz equivalent then  there is a constant $c \in \bb C$ such that 
$$
\begin{cases}
    c^{12}  = 1\\
    c^{3} \in \left\{ \frac{t_2^3}{t_1^3},-\frac{t_2^3}{t_1^3}\right\} 
\end{cases}
$$
For $t_1, t_2$ generic, this is obviously impossible. 
\end{proof}

\begin{example}\label{EX2} Let $f(x,y) =  f_{b,c}(x, y) = x^3 + bx^2y^3 + y^9 + cxy^7$ where $b, c\in \bb C$ are parameters. Then,  $f_{b_1, c_1}$ are not bi-Lipschitz equivalent to $f_{b_2, c_2}$ for $(b_1, c_1)$ and $(b_2, c_2)$ generic. Consequently, $f_{b,c}$ has Lipschitz modality at least $2$.
\end{example}

\begin{proof}

We have 
    $$\partial f/ \partial x = 3x^2 + 2bxy^3 + cy^7,$$  
    $$ \partial f/ \partial y = 3bx^2y^2 + 9 y^8 + 7 cxy^6.$$
Consider the set of all $(b,c) \in \bb C^2$ satisfying $b\neq 0$ and $\frac{4}{27}b^3+ 1\neq 0$ and $c\neq 0$. Using Puiseux's method we find that there are two polar arcs: 
    $$\alpha: x =  -\frac{2b}{3} y^3 + \frac{c}{2b}y^4 + o(y^4) \text{ and } \beta: x = -\frac{c}{2b} y^4 + o(y^4)$$
Computation shows
$$
    f (\alpha(y), y) =\left( \frac{4}{27} b^3 + 1\right) y^9 - \frac{2bc}{3}y^{10} + O(y^{11})
 $$

$$f (\beta(y), y) = y^9    + O(y^{11}),$$

$$ h_0(\alpha) = h_0(\beta) = 9,$$

$$d_{\gr}(\alpha)  =  d_{\gr}(\beta)  = 5,$$

$$ \ord_0(\alpha, \beta) = 3,$$

$$m(\alpha, \beta) = 10 < h_0(\alpha) + \ord_0(\alpha, \beta)  - 1  = 11,$$

$$\nu(\alpha, \beta) = \frac{-2bc}{3(\frac{4}{27}b^3+ 1)} = \frac{-18bc}{4b^3+27}.$$

Then, we have 
$$ \Inv^2(f_{b,c}) = \{ (\frac{4}{27}b^3+ 1)y^9, y^9; (9, \frac{-18bc}{4b^3+27} y^{10}), (9, \frac{18bc}{4b^3+27}) y^{10}\}/\bb C^*.$$

This implies that if $f_{b_1, c_1}$ and $f_{b_2, c_2}$ are bi-Lipschitz equivalent then there is a constant $u \in \bb C^*$ such that 
$$ u^9 \in  \bigg\{ \frac{\frac{4}{27}b_2^3+1}{\frac{4}{27}b_1^3+1},  \frac{1}{\frac{4}{27}b_1^3+1} \bigg\}
$$
and 
$$u \in \bigg\{\frac{b_2c_2 (4b_1^3+27)}{b_1c_1 (4b_2^3+27)}, \frac{-b_2c_2 (4b_1^3+27)}{b_1c_1 (4b_2^3+27)}\bigg\}.$$
This is obviously impossible for $(b_1, c_1)$ and $(b_2, c_2)$ generic. This ends the proof.
\end{proof}

\end{document}